\documentclass{birkjour}

\usepackage{cleveref}
\usepackage{comment}
 \newtheorem{thm}{Theorem}
 \newtheorem{cor}[thm]{Corollary}
 \newtheorem{lem}[thm]{Lemma}
 
 \newtheorem{conj}[thm]{Conjecture}
 \theoremstyle{definition}
 
 \theoremstyle{remark}
 \newtheorem*{rem}{Remark}
 
 \numberwithin{equation}{section}

\begin{document}

%
%
%
%
%
%
%
%
%

\title[Non-Stabilizing Parallel Chip-Firing Games]
 {Non-Stabilizing Parallel Chip-Firing Games}

\author[Ji]{David Ji}

\address{Montgomery High School\\
Skillman, NJ, 08558
}

\email{david.ji@mtsdstudent.us}

\author[Li]{Michael Li}
\address{University High School\br
Irvine, CA, 92612}
\email{michaelli20180101@gmail.com}

\author[Wang]{Daniel Wang}
\address{Lakeside School \br
Seattle, WA, 98125}
\email{danielw25@lakesideschool.org}



\date{\today}


\begin{abstract}
In 2010, Kominers and Kominers proved that any parallel chip-firing game on $G(V,\,E)$ with $|\sigma|\geq 4|E|-|V|$ chips stabilizes. Recently, Bu, Choi, and Xu made the bound exact: all games with $|\sigma|< |E|$ chips or $|\sigma|> 3|E|-|V|$ chips stabilize. Meanwhile, Levine found a ``devil's staircase'' pattern in the plot of the activity of parallel chip-firing games against their density of chips. The stabilizing bound of Bu, Choi, and Xu corresponds to the top and bottom stairs of this staircase, in which the activity is 1 and 0, respectively. In this paper, we analyze the middle stair of the staircase, corresponding to activity $\frac{1}{2}$. We prove that all parallel chip-firing games with $2|E|-|V|< |\sigma|< 2|E|$ have period $T\neq 3,\,4$. In fact, this is exactly the range of $|\sigma|$ for which all games are non-stabilizing. We conjecture that all parallel chip-firing games with $2|E|-|V|< |\sigma|<2|E|$ have $T=2$ and thus activity $\frac{1}{2}$. This conjecture has been proven for trees by Bu, Choi, and Xu, cycles by Dall'asta, and complete graphs by Levine. We extend Levine's method of conjugate configurations to prove the conjecture on complete bipartite graphs $K_{a,a}$. 
\end{abstract}

\maketitle

\section{Introduction}

The \emph{parallel chip-firing game} is an automaton on a simple finite connected graph $G(V,E)$ with vertex set $V$ and edge set $E$. At the beginning of the game, each vertex $v \in V$ holds $\sigma(v)$ \emph{chips}. Each round, all vertices with at least as many chips as neighbors simultaneously \emph{fire} one chip to each of their neighbors. All other vertices \emph{wait}. As in Levine \cite{levine}, we define a \textit{chip configuration} to be the map $\sigma:\, V\rightarrow \mathbb{N}_0$ taking each vertex $v$ to the number of chips on $v$. Let $|\sigma|=\sum_v\sigma(v)$ be the total number of chips. 

Bitar and Goles \cite{bitarandgoles} noted that every parallel chip-firing game is periodic with some minimal period length $T$. Games with $T=1$ are said to \emph{stabilize}. The possible period lengths of parallel chip-firing games on trees \cite{bitarandgoles}, cycles \cite{dallasta}, complete graphs \cite{levine}, and complete bipartite graphs \cite{jiang} have been determined. For general graphs, Kiwi et al. \cite{kiwi} showed that $T$ cannot be bounded by a polynomial in $|V|$. Additionally, Kominers and Kominers \cite{kominers} proved that all games with $|\sigma| \ge 4|E|-|V|$ chips stabilize. Bu, Choi, and Xu \cite{yunseo} made this bound exact, showing that games with $|\sigma| < |E|$ or $|\sigma| > 3|E|-|V|$ chips stabilize. Meanwhile, there exist non-stabilizing parallel chip firing games with $|\sigma|$ chips for each $|E| \le |\sigma| \le 3|E|-|V|$.

The phenomenon of \emph{mode-locking} in parallel chip-firing games was noted by Levine \cite{levine}, who found that there exist regimes of parallel chip-firing games in which adding even a relatively large number of chips does not change the \emph{activity} -- the average fraction of rounds in which a vertex fires. These regimes represent wide stairs in the \emph{devil's staircase} pattern that is revealed when plotting the activity of a parallel chip-firing game against the average value of $\sigma(v)$. The bound of Bu, Choi, and Xu \cite{yunseo} corresponds to the top and bottom stairs of this devil's staircase, in which the activity is 1 and 0, respectively. 

In this paper, we analyze the middle stair of the staircase, corresponding to activity $\frac{1}{2}$. We show that there are no parallel chip-firing games with $2|E|-|V|< |\sigma|< 2|E|$ that have $T = 3$ or $4$. 
\begin{thm}
\label{theorem1}
Any parallel-chip firing game on $G(V,E)$ with $T=3$ or $4$ has $|\sigma| \le 2|E|-|V|$ or $|\sigma| \ge 2|E|$.
\end{thm}
We conjecture from simulating hundreds of thousands of games that any parallel chip-firing game with $2|E|-|V|< |\sigma| < 2|E|$ has $T=2$ and thus activity $\frac{1}{2}$. 
\begin{conj}
    \label{conj1}
    Any parallel chip-firing game on $G(V,E)$ with $2|E|-|V|< |\sigma|< 2|E|$ has $T=2$.
\end{conj}
\begin{rem}
    Bitar and Goles \cite{bitarandgoles} showed that in a stable parallel chip-firing game, either all or no vertices fire. It follows that all stable games have $|\sigma|\leq \sum_v \deg v -1=2|E|-|V|$ or $|\sigma|\geq \sum_v \deg v=2|E|$. Thus, the chip bound in \Cref{conj1} is exactly the bound in which all games are non-stabilizing.
\end{rem}
\Cref{conj1} has been proven for trees by Bu, Choi, and Xu \cite{yunseo}, cycles by Dall'asta \cite{dallasta}, and complete graphs by Levine \cite{levine}. We prove \Cref{conj1} for complete bipartite graphs of the form $K_{a,a}$.
\begin{thm}
\label{theorem2}
    Any parallel-chip firing game on $K_{a,a}$ with $2a^2 - a < |\sigma| < 2a^2$ has $T=2$.
\end{thm}

\section{The Setting}
\label{setting}

We define the \emph{step operator} $U$, where $U\sigma$ is the chip configuration after one round of firing on chip configuration $\sigma$. We let $U^0\sigma = \sigma$ and $U^t\sigma = UU^{t-1}\sigma$. We call $U^t\sigma$ the \emph{chip configuration on round $t$}.

We define the \textit{period} of a parallel chip-firing game to be the smallest $T$ such that there exists a round $t$ with $U^t\sigma = U^{t+T}\sigma$. We denote the smallest such $t$ as $t_0=0$. That is, we consider all parallel chip-firing games to begin on round $t_0$. A \textit{cycle} of a game is defined as the $T$ rounds from round $t+Tn$ to round $t+T(n+1)-1$ for $n,t\geq 0$. 

We define $F_t(v)$ to be $1$ if $v$ fires on round $t$ and $0$ if $v$ waits. Jiang \cite{jiang} showed that all vertices fire the same number of times in a cycle.

\begin{lem}[{\cite[Proposition 2.5]{jiang}}]  \label{constant_firing}
    In any parallel chip-firing game, $U^T\sigma(v)=\sigma(v)$ for all $v\in V$ if and only if $\sum_{t=0}^{T-1} F_t(v)$ is the same for all $v\in V$.
\end{lem}

The \emph{firing sequence} of $v$ is the binary sequence $F_0(v)\dots F_{T-1}(v)$, with indices taken modulo $T$. Jiang, Scully, and Zhang \cite{scully} defined a firing sequence to be \emph{clumpy} if it has both $00$ and $11$ (i.e. 1001 is clumpy). They showed that the firing sequence of a vertex in a parallel chip-firing game is never clumpy.

\begin{lem}[{\cite[Theorem 6.2]{scully}}]  \label{non-clumpy}
    Clumpy firing sequences do not occur in the parallel chip-firing game.
\end{lem}

We call the firing sequence of a vertex $v$ \textit{dense} if for each pair of rounds $0\leq a < b$ with $F_a(v)=F_b(v)=1$, there exists some round $t_u$ with $a < t_u \le b$ for each vertex $u$ in the neighborhood of $v$ such that $F_{t_u}(u) = 1$. 

As defined by Kominers and Kominers \cite{kominers}, we call a vertex $v \in V$ \textit{abundant} on round $t$ if $U^t\sigma(v) \ge 2\deg(v)$. Kominers and Kominers \cite{kominers} showed that abundant vertices do not exist in non-stabilizing games. Jiang \cite{jiang} defined the \emph{complement} of a non-stabilizing parallel chip-firing game on $G=(V,E)$ with chip configuration $\sigma$ as another game on $G$ with chip configuration $\sigma_c$, where $\sigma_c(v)=2\deg(v) - 1 - \sigma(v)$. By definition, we have $|\sigma_c|+|\sigma|=4|E|-|V|$. We let $F^c_t(v)$ be 1 if $v$ fires on round $t$ in the complement game and 0 if it does not. We see that $F^c_0(v)=1$ if and only if $F_0(v)=0$. Jiang \cite{jiang} showed that complement games remain complements in future rounds; it follows that the period of the complement game equals that of the original, and that $F^c_t(v)=1-F_t(v)$ for all $t\geq 0$. 

\section{Main Results}
\subsection{General Graphs}

Let $S_t$ be the set of vertices that fire on round $t$ but not on any prior round $t'$ with $0 \le t'<t \le T-1$. Let $e_t$ be the set of edges with both endpoints in $S_t$ and let $E_{t,t'}$ be the set of edges with one endpoint in $S_t$ and one in $S_{t'}$. Define $e_t(v)$ and $E_{t,t'}(v)$ as the sets of edges in $e_t$ and $E_{t,t'}$, respectively, with $v$ as an endpoint. Finally, let $E_{t} = \bigcup_{m = 0}^{t-2} E_{m,t}$ and let $E_t(v)$ be the set of edges in $E_t$ with $v$ as an endpoint. We take all indices modulo $T$. 

In the following proofs, we assign chips as in Bu, Choi, and Xu \cite{yunseo} such that every chip is assigned to exactly one edge and a chip on vertex $v$ is assigned to an edge with $v$ as an endpoint. Furthermore, each chip is only fired across the edge it is assigned to. A chip assignment is \textit{valid on round 0} if each edge in $E_{0,1}$ has one chip assigned to it. A chip assignment is \textit{valid on round $t$} for $t\geq 1$ if, for all $1\leq k\leq t$, each edge in $e_k$ has one chip assigned to it from each endpoint, each edge in $E_{k-1,k}$ has one or two chips assigned to it, and each edge in $E_{k,k'}$ has one chip assigned to it for all $k'\neq k-1$. Finally, a chip assignment is \emph{valid} if it is valid on round $T$. 

We first show that a valid chip assignment exists for all parallel chip-firing games with $T\geq 3$ in which no vertex fires twice in a row and all firing sequences are dense. Call such a game \emph{compliant}.

\begin{lem} \label{validassignment}
    For any compliant parallel chip-firing game, there exists a valid chip assignment.
\end{lem}
\begin{proof}
We proceed by induction on $t$ to assign the chips held by vertices in $S_t$ on round $t$. As our base case, we assign one chip from each vertex $v\in S_0$ to each edge in $E_{0,1}(v)$. This assignment is valid on round 0. 

We take as our inductive hypothesis that we have a chip assignment that is valid on round $t-1$ for $1\leq t\leq T-1$. 
We keep all of the previous assignments and consider a vertex $u \in S_t$. We have $U^{t-1}\sigma(u)<\deg(u)$. Since all vertices have dense firing sequences, no neighbor of $u$ fires twice before round $t$. Therefore, on round $t-1$, $u$ only receives chips from neighbors in $S_{t-1}$, yielding 
\[
U^t\sigma(u) = |E_{t-1,t}(u)| + U^{t-1}\sigma(u) < |E_{t-1,t}(u)| + \deg(u).
\] 

Since $u$ fires on round $t$, we have $U^t\sigma(u)\geq\deg(u)$. By the inductive hypothesis, there is exactly one chip already assigned to each edge in $E_{t-1,t}(u)$. Furthermore, these chips must belong to $u$ on round $t$, so $u$ holds $U^t\sigma(u)-|E_{t-1,\,t}(u)|\geq \deg(u)-|E_{t-1,t}(u)|$ unassigned chips. We thus assign one chip to each edge that is not in $E_{t-1,t}(u)$. There are fewer than $|E_{t-1,t}(u)|$ remaining chips, which we arbitrarily assign to distinct edges in $E_{t-1,t}(u)$. Each edge in $E_{t-1,t}(u)$ is thus assigned 1 or 2 chips from $u$, and each other edge with $u$ as an endpoint is assigned 1 chip from $u$. Assigning chips for all $u\in S_t$ yields a chip assignment that is valid on round $t$.

We also check that chips can fire across only the edges they are assigned to. Consider a vertex $w$ and a neighboring vertex $w'$ such that $F_t(w)=1$ and $F_m(w)=1$ for some $m<t$ with $F_i(w)=0$ for all $m<i<t$. Then $w$ must have fired a chip assigned to the edge $(w,w')$ to $w'$ on round $m$. Since $w$ has a dense firing sequence, there must exist a round $k$ with $m < k \le t$ such that $F_k(w')=1$. If $m < k < t$, then the chip will have been fired back to $w$ on round $k$. If $k = n$, we imagine first firing the chip from $w'$ to $w$, before firing it right back in the same round. Then it is possible for $w$ to only fire chips across edges they are assigned to.
 
We now assign chips belonging to vertices in $S_T=S_0$ on round $T$. Consider a vertex $v_0 \in S_0$. Since no vertex fires twice in a row, we have $U^{T-1}\sigma(v_0)<\deg(v_0)$. Note that 
\[
U^T\sigma(v_0) = |E_{0, T-1}(v_0)| + U^{T-1}\sigma(v_0) < |E_{0, T-1}(v_0)| + \deg(v_0).
\] 

On round $T$, there is one chip assigned to each edge in $E_{0,k}(v_0)$ for $2 \le k \le T-1$ and one or two chips assigned to each edge in $E_{0,1}(v_0)$. We claim that we can keep these assignments. Consider a vertex $v_m\in S_m$ neighboring $v_0$. Then on round $m$, one or two chips from $v_m$ were assigned to the edge $(v_0,v_m)$. Since chips only fire across edges they are assigned to, these chips must always be on either $v_0$ or $v_m$. By \Cref{constant_firing}, $\sum_{t=0}^{T-1} F_t(v_0)=\sum_{t=0}^{T-1} F_t(v_m)$. Since $v_0$ fired on round 0 and $v_m$ has a dense firing sequence, $\sum_{t=0}^{m-1} F_t(v_0)=1$, while $\sum_{t=0}^{m-1} F_t(v_0)=0$. Then $\sum_{t=m}^{T-1} F_t(v_0)=\sum_{t=m}^{T-1} F_t(v_m)-1$, so exactly one chip assigned to $(v_0,v_m)$ is on $v_0$ on round $T$. Thus, $v_0$ has one chip already assigned to each edge. There are now $U^T\sigma(v_0)-\deg(v_0)<|E_{0, T-1}(v_0)|$ unassigned chips remaining, which we assign arbitrarily to distinct edges in $E_{0, T-1}(v_0)$. Assigning chips for all $v_0\in S_0$ yields a valid chip assignment.
\end{proof}

In the following proofs, we will assume that a valid chip assignment given by \Cref{validassignment} has already been obtained. 

We call an edge \emph{heavy} if two chips are assigned to it. We call other edges \emph{light}. We denote the set of heavy edges in $E_{t,t'}$ as $H_{t,t'}$, and we denote the set of light edges in $E_{t,t'}$ as $L_{t,t'}$. Note that by the definition of a valid assignment, $|H_{t,t'}|=0$ unless $t'=t\pm 1$. Without loss of generality, we let $t'=t+1$. We say that a heavy edge $e=(v,u)\in H_{t,t+1}$ with $v\in S_t$ \emph{leans} $t$ if $v$ sometimes holds both chips assigned to $e$, and that it leans $t+1$ if $u$ sometimes holds both chips. We denote the set of heavy edges in $H_{t,t'}$ that lean $t$ as $H_{t,t'}^t$. We now show that $H_{t,t+1}^{t+1}=H_{t,t+1}$ and $|H_{t,t+1}^t|=0$.

\begin{lem} \label{backheavyset0}
    In a compliant parallel chip-firing game with a valid chip assignment, $H_{t,t+1}^{t+1}=H_{t,t+1}$ and $|H_{t,t+1}^t| = 0$.
\end{lem}
\begin{proof}
    Consider a vertex $v \in S_{t+1}$ and a neighbor $v' \in S_{t}$ such that $(v,v')\in H_{t,t+1}$. The two chips assigned to $(v,v')$ belong to $v$ on round $t+1$. Therefore, $(v,v')$ leans $t+1$. On round $t+1$, one of the two chips is fired to $v'$. Since both $v$ and $v'$ have dense firing sequences, the chip will be returned before $v$ fires again. Thus, $(v,v')$ never leans $t$, so every edge in $H_{t,\,t+1}$ leans $t+1$ and none lean $t$.
\end{proof}

We continue by examining a specific vertex $v\in S_t$. We let $H_{t,t'}(v)$ be the set of edges in $H_{t,t'}$ containing $v$. We similarly define $L_{t,t'}(v)$. We say that $v$ is \emph{deprived} of an edge $e=(v,v')$ on round $k$ if $v$ holds none of the chips assigned to $e$ on round $k$. We now count exactly the number of edges that $v$ is deprived of on round $t-1$.

\begin{lem} \label{deprived}
    In a compliant parallel chip-firing game with a valid chip assignment, a vertex $v \in S_t$ is deprived of exactly $|L_{t-1,t}(v)|$ edges on round $t-1$. 
\end{lem}
\begin{proof}
    We first show that $v$ is deprived of all edges in $L_{t-1,t}(v)$ on round $t-1$. Consider a neighboring vertex $v' \in S_{t-1}$ such that $(v,v')\in L_{t-1,t}(v)$. Since $F_{t-1}(v')=1$, it must hold a chip assigned to each of its edges. Since $(v,v')$ is light, only one chip is assigned to it. This chip must belong to $v'$ on round $t-1$, so $v$ is deprived of all  edges in $L_{t-1,t}(v)$. 

    We next show that $v$ is not deprived of any other edges on round $t-1$. Consider an edge $h\in H_{t-1,t}(v)$. By \Cref{backheavyset0}, $h$ does not lean $t-1$, so at least one chip assigned to $h$ always belongs to $v$. Therefore, $v$ is never deprived of any edges in $H_{t-1,t}(v)$.

    Now consider a vertex $v_k\in S_k$ with $0\leq k\leq t-2$ in the neighborhood of $v$. Then $v_k$ fired a chip $c$ assigned to $(v,v_k)$ to $v$ on round $k$. Since $v\in S_t$, it did not fire on any round before $t$. Thus, $c$ must belong to $v$ on round $t-1$, so $v$ is not deprived of any edges in $E_{t}(v)$ on round $t-1$. 

    Next, consider an edge $e\in e_t(v)$. By the rules of a valid chip assignment, $v$ must always hold one of the chips assigned to $e$, so $v$ is never deprived of any edges in $e_t(v)$.
    
    Finally, consider a vertex $v_j\in S_j$ with $j\geq t+1$ in the neighborhood of $v$. Since $v$ fires before $v_j$, it must be the case that $v$ holds one of the chips assigned to $(v,v_j)$, so $v$ is not deprived of any edges in $E_{t,j}(v)$ for all $j\geq t+1$ on round $t-1$. Thus, $v$ is only deprived of the edges in $L_{t-1,t}(v)$ on round $t-1$.
\end{proof}

We now bound $|L_{t-1,t}(v)|$ using \Cref{deprived}.

\begin{lem} \label{lightedgebound}
    In a compliant parallel chip-firing game with a valid chip assignment, $|L_{t-1,t}(v)| \ge 1$ for any vertex $v \in S_{t}$.
\end{lem}
\begin{proof}
    Note that since $F_{t-1}(v)=0$, we have $U^{t-1}\sigma(v)<\deg(v)$. Since $v$ holds 0, 1, or 2 chips assigned to each of its edges, on round $t-1$, it must be true that $v$ has strictly more edges for which it holds 0 chips than edges for which it holds 2 chips. By Lemma \ref{deprived}, there are $|L_{t-1,t}(v)|$ edges for which $v$ holds 0 chips on round $t-1$. Since the number of edges for which $v$ holds 2 chips is nonnegative, we must have that $|L_{t-1,t}(v)| \geq 1$.
\end{proof}

We are now ready to prove \Cref{theorem1}.

\begin{proof}[Proof of Theorem 1]
    By \Cref{non-clumpy}, in a parallel chip-firing game with $T\geq 3$, either no vertex fires twice in a row or no vertex waits twice in a row. Suppose $G$ is a game with chip configuration $\sigma$ and $T=3$ or 4 in which no vertex fires twice in a row. Then all vertices must have dense firing sequences and $G$ is compliant.
    
    By Lemma \ref{validassignment} and Lemma \ref{lightedgebound}, we have $|L_{t-1,t}(v)| \ge 1$ for all $v \in S_{t}$. Summing over all $v\in V$, we have that the total number of light edges $L$ satisfies
    \[
    L \ge \sum_{t=0}^{T-1}\sum_{v \in S_t}L_{t-1,t}(v) \ge |V|.
    \] 
    
    Since no edge has more than two chips assigned, $|\sigma|\leq 2|E|-|V|$. 
    
    Now suppose $G$ is a game in which no vertex waits twice in a row. Then we consider the complement game of $G$, which has the same period. Since vertices wait in the complement game when they fire in the original, no vertex fires twice in a row in the complement game, and we see that $|\sigma_c|\leq 2|E|-|V|$. Therefore, $|\sigma|\geq 2|E|$. Thus, any parallel chip-firing game with $T=3$ or 4 has $|\sigma| \leq 2|E|-|V|$ or $|\sigma|\geq 2|E|$.
\end{proof}

\subsection{Complete Bipartite Graphs $K_{a,a}$}

We introduce the notion of conjugate configurations, adapted from Levine \cite{levine}. Let the partitions of $K_{a,a}$ be $L$ and $R$. Let the vertices in $L$ be $L_1, \dots, L_a$ with $\sigma(L_i)\geq\sigma(L_{i+1})$ for $1\leq i\leq a-1$. Similarly define $R_1,\dots, R_a$. In this section, we prove results for vertices in $L$. The analogous results hold for vertices in $R$.

We define $\ell_L(\sigma) = \min(\sigma(L_1),\sigma(L_2),\dots, \sigma(L_a))$ as the minimum number of chips on a vertex $L_i$ in $\sigma$ and $r_L(\sigma) = |\{L_i \in L \colon \sigma(L_i) \geq a\}|$ as the number of vertices $L_i$ that fire in $\sigma$. We define $\ell_R$ and $r_R$ analogously. 

The \emph{$j$th conjugate configuration} $c^j\sigma$ of an initial position $\sigma$, for $1\leq j\leq a$, is defined by the following relation for vertices in $L$.
\begin{equation*}
        \begin{aligned}
        &c^j\sigma(L_i) = \sigma(L_i) + j - a ~\text{for}~ i \leq j\\
        &c^j\sigma(L_i) = \sigma(L_i) + j ~\text{for}~ i > j
    \end{aligned}
\end{equation*}

The analogous definitions hold for vertices in $R$. 

Again adopting notation from Levine \cite{levine}, we let $u_t(\sigma,v) = \Sigma_{k=0}^{t-1} F_k(v)$ be the number of times  $v$ fires in the first $t$ rounds from $\sigma$. We let $\alpha_{t}(L)=\sum_{i=1}^au_t(\sigma,\,L_i)$ and $\alpha_{t,\,j}(L)=\sum_{i=1}^au_t(c^j\sigma,\,L_i)$. We also define the analogous values for $R$. Finally, we define the \textit{activity} of a parallel chip-firing game as the function \[A(\sigma) = \lim_{t \rightarrow \infty} \frac{\alpha_t(L) + \alpha_t(R)}{2at}.\] 

A special property holds if $0<A(\sigma)<1$.

\begin{lem} \label{lemma: confined}
    If $0<A(\sigma)<1$, then, for all $t\geq 0$, $$\max(U^t\sigma(L_1),\dots,U^t\sigma(L_a)) - \min(U^t\sigma(L_1),\dots,U^t\sigma(L_a)) < a.$$ 
\end{lem}

\begin{proof}
    Consider a vertex $L_i$. Since $0<A(\sigma)<1$, we must have $T \neq 1$. Thus, $U^t\sigma(L_i) \leq 2a-1$ for all $t \geq 0$. We claim that for any $t \geq 0$, it holds that $r_R(U^t\sigma) \leq U^{t+1}\sigma(L_i) \leq a-1+r_R(U^t\sigma)$. 
    
    Since $L_i$ receives $1$ chip from each of the vertices in $R$ that fire, we have that $r_R(U^{t}\sigma) \leq U^{t+1}\sigma(L_i)$ for all $t$. If $F_{t}(L_i)=0$, we also have that $U^{t+1}\sigma(L_i)=U^{t}\sigma(L_i) + r_R(U^{t}\sigma) \leq a-1+r_R(U^{t}\sigma)$. If $F_{t}(L_i)=1$, we instead have $U^{t+1}\sigma(L_i) = U^{t}\sigma(L_i)-a+r_R(U^{t}\sigma) \leq 2a-1-a+r_R(U^{t}\sigma) = a-1+r_R(U^{t}\sigma)$, yielding the same result.



    Since $U^T\sigma=\sigma$ and $r_R(U^t\sigma) \leq U^{t+1}\sigma(L_i) \leq a-1+r_R(U^t\sigma)$ for all $L_i$ and $t\geq 1$, we must have for all $t\geq 0$ that $$\max(U^t\sigma(L_1),\dots,U^t\sigma(L_a)) - \min(U^t\sigma(L_1),\dots,U^t\sigma(L_a)) < a.\eqno\qed\phantom\qedhere$$ 
\end{proof}

Let $z_t^j(L_i) = u_t(c^j\sigma,L_i) - u_t(\sigma,L_i)$ for all $1\leq i,j\leq a$. Define $z_t^j(R_i)$ analogously. 

\begin{lem} \label{lemma:conj config u}
    For all $t \geq 1$ and $i,j$ with $1\leq i,j\leq a$, we have
\begin{equation*} \label{eq: conj config u ineq}
    \begin{aligned}
    - 1 \leq z_t^j(L_i) \leq 0 ~&\text{if}~ \,i \leq j\text{, and}\\
    0 \leq z_t^j(L_i) \leq 1 ~&\text{if}~ \,i > j.\\
    \end{aligned}
\end{equation*}

The analogous results hold for vertices $R_i$.
\end{lem}

\begin{proof}
We proceed by induction on $t$. We take $t=1$ to be our base case. Note that $c^j\sigma(L_i) \leq \sigma(L_i)$ if $i\leq j$, so if $L_i$ fires in $c^j\sigma$ on turn 0, it also fires in $\sigma$ on turn 0. Thus, $-1 \leq z_1^j(L_i) \leq 0$. Similarly, $c^j\sigma(L_i) \geq \sigma(L_i)$ if $i>j$, so if $L_i$ fires in $\sigma$ on turn 0, it also fires in $c^j\sigma$. Thus, $0 \leq z_1^j(L_i) \leq 1$. 

We assume as our inductive hypothesis that the statement is true for all $1\leq i,j\leq a$ and some particular value of $t$. Note that, for all $1\leq i,j\leq a$,
\begin{align*}
    U^tc^j\sigma(L_i)-U^t\sigma(L_i) &= (c^j\sigma(L_i) + 
    \alpha_{t,j}(R) - au_t(c^j\sigma,L_i))\\
    &\qquad - (\sigma(L_i) + \alpha_t(R) - au_t(\sigma,L_i))\\
    &= c^j\sigma(L_i) - \sigma(L_i) + \alpha_{t,j}(R) - \alpha_t(R) - az_t^j(L_i).
\end{align*}

From the definition of $\alpha_{t,\,j}(R)$,
\begin{align*}
    \alpha_{t,j}(R) - \alpha_t(R) &= \sum_{i=1}^a [u_t(c^j\sigma, R_i) - u_t(\sigma, R_i)]=\sum_{i=1}^a z_t^j(R_i). 
\end{align*}

By the inductive hypothesis, \[-j \leq \alpha_{t,j}(R) - \alpha_t(R) \leq a-j.\]

There are now four cases.

\underline{Case 1:} $i \le j$ and $z_t^j(L_i) = 0$.

We must have that
\[U^tc^j\sigma(v)-U^t\sigma(v) = j-a + \alpha_{t,j}(R) - \alpha_t(R) \leq 0.\]

Therefore, if $L_i$ fires in $c^j\sigma$ on turn $t$, it also fires in $\sigma$ on turn $t$. Thus, $-1 \leq z_{t+1}^j(L_i)-z_t^j(L_i) \leq 0$, so $-1 \leq z_{t+1}^j(L_i) \leq 0$.

\underline{Case 2:} $i \le j$ and $z_t^j(L_i) = -1$.

We must have that
\[U^tc^j\sigma(L_i)-U^t\sigma(L_i) = j-a + \alpha_{t,j}(R) - \alpha_t(R) + a = j + \alpha_{t,j}(R) - \alpha_t(R) \geq 0.\]

Therefore, if $L_i$ fires in $\sigma$ on turn $t$, it also fires in $c^j\sigma$ on turn $t$. Thus, $0 \leq z_{t+1}^j(L_i)-z_t^j(L_i) \leq 1$, so $-1 \leq z_{t+1}^j(L_i) \leq 0$.

\underline{Case 3:} $i > j$ and $z_t^j(L_i) = 0$.

We must have that
\[U^tc^j\sigma(L_i)-U^t\sigma(L_i) = j + \alpha_{t,j}(R) - \alpha_t(R) \geq 0.\]

Therefore, if $L_i$ fires in $\sigma$ on turn $t$, it also fires in $c^j\sigma$ on turn $t$. Thus, $0 \leq z_{t+1}^j(L_i)-z_t^j(L_i) \leq 1$, so $0 \leq z_{t+1}^j(L_i) \leq 1$.

\underline{Case 4:} $i > j$ and $z_t^j(L_i) = 1$. 

We must have that
\[U^tc^j\sigma(L_i)-U^t\sigma(L_i) = j + \alpha_{t,j}(R) - \alpha_t(R) - a \leq 0.\]

Therefore, if $L_i$ fires in $c^j\sigma$ on turn $t$, it also fires in $\sigma$ on turn $t$. Thus, $-1 \leq z_{t+1}^j(L_i)-z_t^j(L_i) \leq 0$, so $0 \leq z_{t+1}^j(L_i) \leq 1$.

The inductive step holds in all cases, so the proof is complete.
\end{proof}

From \Cref{lemma:conj config u} we see the following corollary.

\begin{cor}\label{cor: conj config a equal}
    For all $j$ with $1\leq j\leq a$, it holds that $A(\sigma)=A(c^j\sigma)$.
\end{cor}

The following lemma relates the round-dependent function $u_{2t}(\sigma,v)$ to the round-independent function $u_2(\sigma,v)$.

\begin{lem} \label{lemma: u2 and u2t}
    If $u_2(\sigma,L_i) \ge 1$ for all $i$ with  $1\leq i \leq a$, then $u_{2t}(\sigma,L_i) \ge t$ for all $i$ and $t \ge 1$. The analogous results hold for vertices $R_i$.
\end{lem}

\begin{proof}
We proceed by induction on $t$. The base case is given, and we take as our inductive hypothesis that $u_{2t}(\sigma,L_i)\geq t$ and $u_{2t}(\sigma,R_i)\geq t$ for all $i$. Consider a vertex $L_i$. Since $u_{2t+2}(\sigma, L_i) \geq u_{2t}(\sigma, L_i)$, we assume that $u_{2t}(\sigma, L_i)<t+1$. Otherwise, the inductive step is automatically complete. Also, since $u_2(\sigma,L_i) \ge 1$, we must have $\sigma(L_i) \ge a$ or $U\sigma(L_i) \ge a$. 

\underline{Case 1:} $\sigma(L_i) \ge a$.

Since $\alpha_{2t}(R)= \sum_{i=1}^a u_{2t}(\sigma,R_i) \geq at$, we have that
\[U^{2t}\sigma(L_i) = \sigma(L_i) + \alpha_{2t}(R) - au_{2t}(\sigma,L_i) \geq \sigma(L_i)\geq a.\] 

Therefore, $U^{2t}\sigma(L_i) \ge a$, so $F_{2t}(L_i)=1$. Thus, $u_{2t+1}(\sigma,L_i)\geq t+1$.

\underline{Case 2:} $U\sigma(L_i) \ge a$ and $\sigma(L_i) < a$.

Suppose a vertex $R_i$ satisfies $F_0(R_i)=1$. Then by Case 1, we have $u_{2t+1}(\sigma,R_i)\geq t+1$. All other vertices $R_j$ obey $u_{2t+1}(\sigma,R_j)\geq t$ by the inductive hypothesis. The number of vertices in $R$ that fired on round 0 is given by $\alpha_0(R)$, so we have $$\alpha_{2t+1}(R) \geq \alpha_0(R)(t+1)+(a-\alpha_0(R))(t)=\alpha_0(R)+at.$$

Now suppose $F_{2t}(L_i)=0$. Then $u_{2t+1}(\sigma,L_i)=u_{2t}(\sigma,L_i)=t$. Thus, $$U^{2t+1}\sigma(L_i)-U\sigma(L_i)=\alpha_{2t+1}(R)-\alpha_0(R)-at \geq 0.$$ 

Therefore, $U^{2t+1}\sigma(L_i)\geq U\sigma(L_i)\geq a$, so $F_{2t+1}=1$. Thus, we must have $u_{2t+2}(\sigma,L_i)\geq t+1$ in both cases.
\end{proof}

Using Lemmas  \ref{lemma: confined} and \ref{lemma: u2 and u2t}, and \Cref{cor: conj config a equal}, we prove \Cref{theorem2}.

\begin{proof} [Proof of Theorem \ref{theorem2}]
    Consider a parallel chip-firing game on a complete bipartite graph $K_{a,a}$ with $2a^2-2a < |\sigma| < 2a^2$. Let $|\sigma_L|=\sum_{i=1}^a\sigma(L_i)$ and $|\sigma_R|=\sum_{i=1}^a\sigma(R_i)$ denote the total number of chips belonging to vertices in $L$ and $R$, respectively. Without loss of generality, let $|\sigma_L|\leq |\sigma_R|$.

    Since $2a^2-2a<|\sigma|<2a^2$, at least one vertex fires and at least one vertex waits each turn. Thus, $0<\frac{1}{2a}\leq A(\sigma)\leq\frac{2a-1}{2a}<1$. By \Cref{lemma: confined}, for all $1 \leq i < j \leq a$, we must have $\sigma(R_i) - \sigma(R_j) < a$. Then \[\sum_{j=1}^{a}\ell_R(c^j\sigma) = \sum_{j=1}^{a} \sigma(L_R)+j-a = |\sigma_R| + \frac{a(a+1)}{2}-a^2 = |\sigma_R| - \frac{a(a-1)}{2}.\]

 We also have $a^2 > |\sigma_L| \geq (\sigma(L_i)-a+1)(a-i)+\sigma(L_i)(i)$. Rearranging,
    \begin{equation} \label{eq: sigma bound}
        \begin{aligned}
        \sigma(L_i) < \frac{a^2+(a-1)(a-i)}{a} \leq 2a-i.
        \end{aligned}
    \end{equation}

    For a fixed vertex $L_i$, as $j$ ranges from $1$ to $a$, the value $c^j\sigma(L_i)$ takes on each of the values $\sigma(L_i)+i-a,\sigma(L_i)+i+1-a,\dots,\sigma(L_i)+i-1$ exactly once. By \Cref{eq: sigma bound}, $c^j\sigma(L_i) \geq a$ for at least $(\sigma(L_i)+i-1)-a+1=\sigma(L_i)+i-a\leq a$ values of $j$. (If $\sigma(L_i)+i-a\leq0$, we have $c^j\sigma(L_i) \geq a$ for 0 values of $j$.)  Thus,  \[\sum_{j=1}^ar_L(c^j\sigma) \geq \sum_{i=1}^a \sigma(L_i)+i-a = |\sigma_L|+\frac{a(a+1)}{2}-a^2 = |\sigma_L|-\frac{a(a-1)}{2}.\] 
    
    Note now that \[\sum_{j=1}^a \ell_R(c^j\sigma) + r_L(c^j\sigma) \geq |\sigma|-a(a-1) > a^2-a.\] 
    
    By the pigeonhole principle, we must have $\ell_R(c^j\sigma) + r_L(c^j\sigma) \geq a$ for some value of $j$. In $c^j\sigma$, every vertex in $R$ fires on at least one of turns $0$ and $1$, so by \Cref{lemma: u2 and u2t}, we have that $u_{2t}(c^j\sigma,R_i) \geq t$ for all $t$. Also, since every vertex in $R$ fires on at least one of turns $0$ and $1$, every vertex in $L$ fires on at least one of turns $1$ and $2$. After an index shift, we use \Cref{lemma: u2 and u2t} to conclude that $u_{2t+1}(c^j\sigma,L_i) \geq t$ for all $t$. By \Cref{cor: conj config a equal}, $A(\sigma) \geq \frac{1}{2}$.

    The same argument holds for the complement game $\sigma_c$, which has $|\sigma_c|=4a^2-2a-|\sigma| < 2a^2$. Thus, $A(\sigma_c) \geq \frac{1}{2}$. Since a vertex fires in the complement game exactly when it waits in the original, $A(\sigma) \leq \frac{1}{2}$.

    Therefore, $A(\sigma)=\frac{1}{2}$. By \Cref{non-clumpy}, the only possible firing sequences are 10 and 01. Thus, $T=2$.
\end{proof}


\end{document}